\documentclass[11pt,reqno,a4paper]{amsart}
 \usepackage{amsgen, amstext,amsbsy,amsopn, amsthm, amsfonts,amssymb,amscd,amsmath,euscript,enumerate,url,verbatim,calc,tikz}
\usepackage{hyperref}
\usepackage{MnSymbol}
\usepackage{pgfkeys}
\usepackage{mathtools}
\usetikzlibrary{calc}
\usepackage{etex}
 \RequirePackage{etex}
\usepackage[left=1.2in,right=1.2in,bottom=1.5in]{geometry}

\usepackage{thmtools, thm-restate}
\usepackage{pst-node}
\usepackage{array}
\usepackage{mathrsfs}
\usepackage{tikz-cd}

\usepackage{tikz}
\usepackage{tikz-network}
\usepackage{eqnarray,amsmath}

\let\emptyset\varnothing

\def\multiset#1#2{\ensuremath{\left(\kern-.3em\left(\genfrac{}{}{0pt}{}{#1}{#2}\right)\kern-.3em\right)}}

\usetikzlibrary{arrows}

 \usepackage{latexsym}
 \usepackage{graphics}
 \usepackage{color}
\usepackage{lastpage}
\usepackage{fancyhdr}
\usepackage{multirow}
\usepackage[capitalise]{cleveref}
\allowdisplaybreaks
\usepackage{graphicx}
\graphicspath{ {F:/IMAGES/} }

\makeatletter
\def\oversortoftilde#1{\mathop{\vbox{\m@th\ialign{##\crcr\noalign{\kern3\p@}%
      \sortoftildefill\crcr\noalign{\kern3\p@\nointerlineskip}%
      $\hfil\displaystyle{#1}\hfil$\crcr}}}\limits}

\def\sortoftildefill{$\m@th \setbox\z@\hbox{$\braceld$}%
  \braceld\leaders\vrule \@height\ht\z@ \@depth\z@\hfill\braceru$}

\makeatother

 \newcommand{\p}{\mathfrak{p}}
 \newcommand{\la}{\langle}
 \newcommand{\ra}{\rangle}

 \newcommand{\RV}{\mathcal{RV}}

  \newcommand{\N}{\mathbb{N}}

\def\v{\mathrm{v}}

  \newcommand{\pd}{\operatorname{pd}}

  \newcommand{\Ass}{\operatorname{Ass}}

  \newcommand{\reg}{\operatorname{reg}}
  \newcommand{\im}{\operatorname{im}}
  \newcommand{\mat}{\operatorname{m}}

 \newcommand{\depth}{\operatorname{depth}}

\newcommand{\proset}{\,\mathrel{\lower 4pt\hbox{$\scriptscriptstyle/$}
\mkern -14mu\subseteq }\,} 

 \newtheorem{theorem}{Theorem}[section]
 \newtheorem{lemma}[theorem]{Lemma}
 \newtheorem{proposition}[theorem]{Proposition}

 \newtheorem{conjecture}[theorem]{Conjecture}

\usepackage{amsmath}

 \theoremstyle{definition}
 
 \newtheorem{remark}[theorem]{Remark}
 \newtheorem{definition}[theorem]{Definition}

\newtheorem{innercustomthm}{Theorem}
\newenvironment{customthm}[1]
  {\renewcommand\theinnercustomthm{#1}\innercustomthm\itshape}
  {\endinnercustomthm}

\title[Lattice points arising from regularity and $\mathrm{v}$-number]{Lattice points arising from regularity and $\mathrm{v}$-number of Graphs: Whisker and Cameron-Walker}
\author{Prativa Biswas, Mousumi Mandal, and Kamalesh Saha}

\thanks{AMS Classification 2020: 05C69, 05C25, 05E40, 13D02, 13F55}
\address{Department of Mathematics, Indian Institute of Technology Kharagpur, 721302, India}\email{prativabiswassnts@kgpian.iitkgp.ac.in}
\address{Department of Mathematics, Indian Institute of Technology Kharagpur, 721302, India}\email{mousumi@maths.iitkgp.ac.in}
\address{Department of Mathematics, SRM University-AP, Amaravati 522240, Andhra Pradesh, India}\email{kamalesh.saha44@gmail.com; kamalesh.s@srmap.edu.in}
\begin{document}

\begin{abstract}
Let $G$ be a simple graph on $n$ vertices and $I(G)\subseteq R$ be its edge ideal. In this paper, we initiate the study of determining lattice points in $\mathbb{N}^2$ that appear as a pair $(\mathrm{reg}(R/I(G)), \mathrm{v}(I(G)))$, where $G$ ranges over all connected graphs on $n$ vertices, and we denote this set by $\mathcal{RV}(n)$. Here `$\mathrm{reg}$' denotes the (Castelnuovo-Mumford) regularity and `$\mathrm{v}$' denotes the $\mathrm{v}$-number. We establish general bounds for $\RV(n)$ by identifying two sets $A(n)$ and $B(n)$ satisfying $A(n)\subseteq \mathcal{RV}(n)\subseteq B(n)$. Furthermore, we explicitly determine the subsets of $\mathcal{RV}(n)$ consisting of all possible pairs $(\mathrm{reg}(R/I(G)), \mathrm{v}(I(G)))$ arising from whisker graphs and Cameron-Walker graphs on $n$ vertices. Finally, we propose a conjecture on the subset of $\mathcal{RV}(n)$ arising from connected chordal graphs.
\end{abstract}

\maketitle

\section{Introduction}

In graph theory, a classical problem is to determine the set of lattice points arising as tuples of graph invariants for connected graphs on a fixed number of vertices. More concretely, let $a_1(G), \ldots, a_k(G)$ be $k$ invariants associated with a graph $G$. One then investigates the collection of all possible tuples 
\[
(a_1(G), \ldots, a_k(G)),
\]
where $G$ ranges over all connected graphs on a fixed number of vertices $n$. Essentially, understanding these boundaries helps determine which combinations of graph invariants are possible and prevents us from searching for graphs that cannot exist.


Since the introduction of edge ideals by Villarreal \cite{vil90} in 1990, this perspective has attracted significant attention in commutative algebra as well. Researchers study algebraic invariants of edge ideals in a similar spirit, and this has now become a classical problem in commutative algebra. For instance, considering the set of all possible pairs $(\reg(R/I(G)), \pd(R/I(G)))$ for connected graphs $G$ with $|V(G)| = n$ provides insight into the possible shapes and sizes of the Betti tables of the modules $R/I(G)$. Here $I(G)$ denotes the edge ideal of $G$, $R=K[V(G)]$ with $K$ a field, $\reg$ denotes the (Castelnuovo-Mumford) regularity, and $\pd$ stands for projective dimension.\par 

In \cite{hh21}, Ha-Hibi initiated the investigation of the pair $(\pd(R/I(G)),\reg(R/I(G)))$ and later Erey-Hibi in \cite{eh22} determine complete list of these pairs for bipartite graphs. In \cite{hkkmv21}, Hibi-Kanno-Kimura-Matsuda-Van Tuyl completely determine the lattice points arising as the tuples $(\depth(R/I(G)),\reg(R/I(G)),\dim(R/I(G)),\deg (h(R/I(G))))$ for Cameron-Walker graphs on $n$ vertices. In \cite{hkmt21}, Hibi-Kimura-Matsuda-Tsuchiya studied the pair $(\reg(R/I(G)),a(R/I(G)))$ for Cameron-Walker graphs. In \cite{hkmv22}, Hibi-Kimura-Matsuda-Van Tuyl studied the pair $(\reg(R/I(G)),\deg(h(R/I(G))))$ and determine these pairs completely for Cameron-Walker graphs. In \cite{hmadani23}, Hibi-Madani investigated the triplet $(\kappa(G), f(G), diam(G))$ for connected graphs on $n$ vertices. In \cite{hku23}, Higashitani-Kanno-Ueji investigated the possible pairs $(\dim(I(G)),\depth(I(G)))$ for connected graphs on $n$ vertices.\par 

In 2020, motivated by the study of Reed–Muller type codes, Cooper et al. introduced a new invariant of ideals, called the $\v$-number, named after Wolmer Vasconcelos (also referred to as the Vasconcelos number). Since its introduction, the $\v$-number has attracted considerable attention, particularly in the context of edge ideals (see \cite{bm25, bms24, civan23, ficsimon25, fm25, grv21, jv21, kmt25, kns25, sahacover23, saha-binexpv, ksvgor23, ss22, sahavan25}). Notably, Saha–Sengupta \cite{ss22} and Civan \cite{civan23} established that no general relationship exists between the $\v$-number and the Castelnuovo–Mumford regularity of edge ideals. In particular, it is possible for $\reg(R/I(G))$ to be arbitrarily larger than $\v(I(G))$, and conversely, for $\v(I(G))$ to significantly exceed $\reg(R/I(G))$ even for connected graphs. For simplicity of notation, we write $\reg(G)$ to mean $\reg(R/I(G))$, and $\v(G)$ instead of $\v(I(G))$. We denote the set of all possible pairs $(\reg(G),\v(G))$, where $G$ ranges over all connected graphs on $n$ vertices, by $\RV(n)$, i.e.,
$$\RV(n)=\left\{ (r,v)\in \mathbb{N}^2 \;\middle|\; 
\begin{array}{l}
      \text{there exists a connected graph $G$ on $n$ vertices}\\
      \hspace{1cm}\text{ with } \reg(G)=r \text{ and } \v(G)=v
\end{array}
     \right\}.$$

The observations in the previous paragraph naturally lead to the problem of determining the set of all possible pairs $(\reg(G), \v(G))$ arising from connected graphs on $n$ vertices, that is, determining the set $\RV(n)$. Moreover, even after knowing $\RV(n)$, it is important and challenging to find subsets of $\RV(n)$ when we restrict ourselves to some important classes of graphs in the context of edge ideals, such as chordal, bipartite, whisker, Cameron-Walker, etc. To date, no systematic study has been carried out in this direction, and no general method is available for handling the $\v$-number compared to other invariants for graphs with a fixed number of vertices. In this paper, we investigate this problem in detail and propose approaches that may pave the way for future developments in the study of the $\v$-number. The details of the work done in this paper are described below section-wise.

In \Cref{preli}, we recall the preliminaries required for this paper. \Cref{sec:3} is devoted to study the lattice points in $\N^2$ appearing as pairs of (regularity,$\v$-number) of connected graphs on a fixed number of vertices. First, in \Cref{prop:v<=mat}, we establish an upper bound for $\v(G)$ using the edge domination number of $G$. Then, for a given positive integer $n\geq 3$, we find two sets $A(n)$ and $B(n)$ in $\N^2$ which gives a prediction about our desired lattice points. The main result of this section is the following.
\begin{customthm}{\ref{thm:RV(n)}}
    Let $A(n)=\{(r,v)\in \mathbb{N}^{2} \mid 1\leq r<\frac{n}{2}, 1\leq v\leq r-\lceil\frac{r}{n-2r}\rceil+1\}$ and $B(n)=\{(r,v)\in\mathbb{N}^2 \mid 1\leq r<\frac{n}{2}, 1\leq v<\frac{n}{2}\}$. Then for $n\geq 3$, we have
    $$A(n)\subseteq \mathcal{RV}(n)\subseteq B(n).$$
\end{customthm}
\noindent Next, in \Cref{sec:whisker graph}, we completely determine all possible pairs (regularity,$\v$-number) arising from connected whisker graphs on a fixed number of vertices. Specifically, we find the following subset of $\RV(n)$, which is a restriction of $\RV(n)$ to whisker graphs, denoted by $\RV_{W}(n)$.
\begin{customthm}{\ref{thm:wrv}}
    For the class of connected whisker graphs $W_G$ on $n=2m$ vertices, the lattice points of $(\reg(W_G),\v(W_G))$ are given by
$$\mathcal{RV}_{W}(n)=\{(r,v)\mid 1\leq r\leq m-1 \text{ and } 1\leq v\leq r-\left\lceil\frac{r}{m-r}\right\rceil+1\}.$$
Moreover, since a whisker graph always has even number of vertices, we have $\RV_{W}(n)=\emptyset$ if $n$ is odd.
\end{customthm}
\noindent In the final section (\Cref{sec:CW}), we carry out a study similar to that for whisker graphs, focusing on Cameron–Walker graphs. Precisely, we determine the set $\RV_{CW}(n)$, which is the restriction of the set $\RV(n)$ to Cameron-Walker graphs. A classification of the lattice points in $\mathbb{N}^2$ arising as pairs (regularity, $\v$-number) of Cameron–Walker graphs on a fixed number of vertices is presented below.
\begin{customthm}{\ref{thm:cwrv}}
     For the class of Cameron-Walker graphs, we have $\mathcal{RV}_{CW}(n)=\emptyset$ for $n<5$, and for $n\geq 5$, we have
     $$\mathcal{RV}_{CW}(n)=\left\{(r,v)\mid 2\leq r\leq \left\lceil\frac{n-1}{2}\right\rceil \text{ and } 1\leq v\leq \min\{r-1,n-2r\}\right\}.$$
\end{customthm}
\noindent We conclude this article with a conjecture on the subset of $\RV(n)$ that arises from connected chordal graphs, along with some possible directions for future research.

\section{Preliminaries}\label{preli}
In this section, we recall some of the relevant prerequisites on graph theory and commutative algebra.

\begin{definition}
     A (simple) graph $G$ is a pair $(V(G),E(G))$, where $V(G)$ is a finite set, called the vertex set of $G$, and $E(G)$ is a collection of cardinality two subsets of $G$, known as the edge set of $G$. The edge ideal of $G$, denoted by $I(G)$, is a quadratic square-free monomial ideal of $R=K[V(G)]$ defined as 
     $$I(G):=\la x_{i}x_{j}\mid \{x_i,x_j\}\in E(G)\ra.$$
\end{definition}

A graph $H$ is called a \textit{subgraph} of a graph $G$ if $V(H)\subseteq V(G) $ and $E(H)\subseteq E(G)$. A subgraph $H$ of $G$ is said to be an \textit{induced subgraph} of $G$ if for any two vertices $u,v\in V(H)$, $\{u,v\}\in E(H)$ whenever $\{u,v\}\in E(G)$. For any $A\subseteq V(G)$, we denote by $G[A]$ the induced subgraph of $G$ on the vertex set $A$. Also, we write $G-A$ to denote the induced subgraph $G[V(G)\setminus A]$. If $A=\{x\}$ consists of a single vertex, then we simply write $G-x$ instead of $G-\{x\}$. The \textit{neighbor} set of $A$ in $G$, denoted by $N_{G}(A)$, is defined by
$$N_{G}(A)=\{x_i\in V(G)\mid \{x_i,x_j\}\in E(G) \text{ for some }x_j\in A\}.$$
We write $N_{G}[A]:=N_{G}(A)\bigcup A$. Again, for any two subsets $A,B$ of $V(G)$, we define $N_{B}(A)$ as the neighbors of $A$ in $B$, i.e., $N_{B}(A)=N_{G}(A)\bigcap B$. The \textit{degree} of a vertex $x$ in $G$, denoted by $\deg_{G}(x)$, is the cardinality of the set $N_{G}(\{x\})$.

A set $A\subseteq V(G)$ is said to be an independent set of $G$ if the graph $G[A]$ contains no edge, in other words, there is no edge between any two vertices in $A$. The maximum size among all independent sets in $G$ is known as \textit{independence} number of $G$, denoted by $\alpha(G)$. A set $C\subseteq V(G)$ is said to be a vertex cover of $G$ if $C\bigcap e\neq \emptyset$ for every $e\in E(G)$. By these two definitions, one can verify that $A$ is a maximal independent set of $G$ if and only if $N_{G}(A)=V(G)\setminus A$ is a minimal vertex cover of $G$. Note that if $A$ is an independent set, then $N_{G}(A)\bigcap A=\emptyset$.

\begin{definition}
    Let $G$ be a graph and $M$ be a set of pairwise disjoint edges of $G$. Then $M$ is called a \textit{matching} of $G$. The matching number of $G$ is the number of edges in a maximum matching of $G$ and is denoted by $\mat(G)$. An \textit{induced matching} of the graph $G$ is a matching $M=\{e_1,\ldots,e_m\}$ of $G$ such that the only edges of $G$ contained in $\bigcup_{i=1}^{m}{e_i}$ are $e_1,\ldots,e_m$. The \textit{induced matching number} of $G$ is the number of edges in a maximum induced matching of $G$ and is denoted by $\mbox{im}({G})$.
\end{definition}

\begin{definition}
    Let $R$ be a polynomial ring with standard grading and $I\subseteq R$ be a graded ideal. Then for any $\p\in \Ass(I)$ there exists a homogeneous polynomial $f$ such that $(I:f)=\p$, where $\Ass(I)$ denote the set of associated primes of $I$. The $\v$-number of $I$ is defined as follows:
    $$\v(I):=\min\{\deg(f)\mid \text{ $f$ is homogeneous and $(I:f)$ is a prime}\}.$$
\end{definition}

Generally, the invariant $\v$-number depends on the characteristic of the base field $K$. However, Jaramillo and Villarreal showed in \cite{jv21} that the $\v$-number is independent of the choice of the base field for square-free monomial ideals by providing a combinatorial description of the $\v$-number. In particular, thanks to \cite[Lemma 3.4]{jv21} and \cite[Theorem 3.5]{jv21}, we get the following combinatorial definition of the $\v$-number in case of the edge ideal of a graph or simply, the $\v$-number of a graph.

\begin{definition}\label{def:v}
Let $G$ be a graph and $A$ be an independent set of $G$ such that $N_{G}(A)$ is a vertex cover of $G$. Then we have $(I(G):X_{A})=\la N_{G}(A)\ra$, where $X_A=\prod_{x_i\in A}x_i$. Moreover, we have
    $$\v(G):=\min\{|A| \,\mid\, \text{$A$ is independent and $N_{G}(A)$ is a vertex cover of $G$}\}.$$
    If $A$ is an independent set such that $N_{G}(A)$ is a vertex cover of $G$ and $|A|=\v(G)$, then we say $A$ is an independent set of $G$ corresponding to $\v(G)$.
\end{definition}

Now, let us recall some well-known classes of graphs.
\begin{definition}
    A \textit{path} graph on $n$ vertices, denoted by $P_n$, is a graph such that after relabeling the vertices we have $V(P_n)=\{x_1,\ldots,x_n\}$ and $E(P_n)=\{\{x_{i},x_{i+1}\}\mid 1\leq i\leq n-1\}$. A graph on $n$ vertices is called a \textit{cycle} of length $n$, denoted by $C_n$, if it is connected and every vertex has exactly two neighbors. If a connected graph contains no induced cycle, then we call it a \textit{tree}, and a disjoint union of trees is called a \textit{forest}. A graph $G$ is called \textit{chordal} if $G$ has no induced cycle of length more than three. A \textit{bipartite} graph is a graph without any induced cycle of odd length, in other words, a graph is said to be bipartite if the vertex set can be partitioned into two independent sets. A \textit{complete} graph on $n$ vertices, denoted by $K_n$, is a graph such that there is an edge between every pair of vertices. 
\end{definition}

 Let $M$ be a finitely generated graded $R$-module. Then we can write the graded minimal free resolution of $M$ in the following form:
    $$0\rightarrow\bigoplus_{j\in\mathbb{Z}} R(-j)^{\beta_{g,j}(M)}\rightarrow\cdot\cdot\cdot\rightarrow\bigoplus_{j\in\mathbb{Z}} R(-j)^{\beta_{1,j}(M)}\rightarrow \bigoplus_{j\in\mathbb{Z}} R(-j)^{\beta_{0,j}(M)}\rightarrow M\rightarrow 0,$$
    where $\beta_{i,j}(M)\in \mathbb{Z}_{\geq 0}$ is called the $(i,j)$-th graded \textit{Betti} number of $M$ and $R(-j)$ denotes the polynomial ring shifted in degree $j$.
    
\begin{definition}
    The \textit{Castlenuovo-Mumford regularity} (or \textit{regularity} in short) of $M$, denoted by $\reg(M)$, is defined as follows 
    $$\reg(M):=\max\{j-i\mid \beta_{i,j}(M)\neq 0\}.$$
\end{definition}

\begin{lemma}\cite[Lemma 2.10]{dhs13}\label{Dao}
Let $G$ be a simple graph. Then for any vertex $x$ of $G$,
$$\reg(G)\leq \max\{\reg(G-x),~\reg(G-N_G[x])+1\}.$$   
\end{lemma}
Next we have available bounds of regularity and $\mathrm{v}$-number for graphs from the literature:
\begin{theorem}\cite[Lemma 2.2]{katzman06}\label{thm:reg lower}
 For any graph $G$, $\reg(G)\geq \im(G)$.
\end{theorem}       

\begin{proposition}\label{prop:reg<n/2}
    Let $G$ be a connected graph on $n\geq 3$ vertices. Then $\reg(G)<\frac{n}{2}$.
\end{proposition}
\begin{proof}
    By \cite[Theorem 6.7]{hv08}, we have $\reg(G)\leq m(G)$. From the definition of matching number, we see that $m(G)\leq \frac{n}{2}$. If $n$ is odd, then we obviously have $m(G)<\frac{n}{2}$, and consequently, $\reg(G)<\frac{n}{2}$. If $n$ is even, then $n\geq 4$, and thus, we have $\reg(G)<\frac{n}{2}$ by \cite[Corollary 3.3]{eh22}.
\end{proof}

\begin{theorem}\cite[Theorem 2.18]{zheng04}\label{thm:reg forest}
  If $G$ is a forest, then $\reg(G)= \im(G)$.  
\end{theorem}

\begin{theorem}\cite[Corollary 6.9]{hv08}\label{thm:reg chordal}
    For a chordal graph $G$, we have $\reg(G)=\im(G)$.
\end{theorem}

\section{An approximation of the set $\RV(n)$}\label{sec:3}
In this section, our main aim is to give an approximate range of the set $\RV(n)$. We first establish a sharp upper bound of $\v(G)$ using edge domination number of $G$ and the number of vertices in $G$. Finally, we identify two subsets $A(n)$ and $B(n)$ of $\mathbb{N}^2$, and show that $A(n)\subseteq \RV(n)\subseteq B(n)$. Let us start with the definition of edge domination number of a graph.

\begin{definition}
    The minimum size of a maximal matching in a graph $G$ is known as the \textit{edge domination number} of $G$, which is denoted by $\gamma_{e}(G)$. Thus, we have $\gamma_{e}(G)\leq m(G)$ for any graph $G$.
\end{definition}
The following gives a relation between the $\v$-number and edge domination number.

\begin{proposition}\label{prop:v<=mat}
    Let $G$ be a connected graph on $n\geq 3$ vertices. Then 
    $$\v(G)\leq \min\left\{\gamma_{e}(G),\left\lfloor\frac{n-1}{2}\right\rfloor\right\}.$$
    Hence, we have $\v(G)\leq \min\{m(G),\lfloor\frac{n-1}{2}\rfloor\}$.
\end{proposition}
\begin{proof}
    Let $M$ be a maximal matching in $G$ with $\vert M\vert=\gamma_{e}(G)$. Let us consider a maximal independent set of vertices $A$ of $G$ inside $\bigcup_{e\in M}e$. Now, let $e'=\{x_i,x_j\}$ be any edge of $G$. Then one of $x_i$ and $x_j$ belongs to $\bigcup_{e\in M}e$ as $M$ is a maximal matching in $G$. Without loss of generality, we assume $x_{i}\in \bigcup_{e\in M}e$. If $x_i\in A$, then $x_j\in N_{G}(A)$. Again, if $x_i\not\in A$, then $x_i\in N_{G}(A)$ as $A$ is a maximal independent set inside $\bigcup_{e\in M}e$. Thus, $N_{G}(A)$ is a vertex cover of $G$. Since $A$ is an independent set of $G$ and $N_{G}(A)$ is a vertex cover of $G$, by the combinatorial definition of the $\v$-number of graphs (see \Cref{def:v}), we have $\v(G)\leq \vert A\vert$. Since $A$ is an independent set of $G$ inside $\bigcup_{e\in M}e$, we have $|A|\leq |M|=\gamma_{e}(G)$, and hence, $\v(G)\leq \gamma_{e}(G)$.\par 

    Now our aim is to prove that $\v(G)\leq \lfloor\frac{n-1}{2}\rfloor$. Note that if $n$ is odd, then $m(G)\leq \lfloor \frac{n-1}{2}\rfloor$, and thus, the inequality follows from the previous inequality. So, the only case that remains to be considered is when $n$ is even and $m(G)=\frac{n}{2}$. Here the connectedness of $G$ and $n\geq 3$ part will come into play. Let $M=\{e_1,\ldots,e_m\}$ be a maximal matching in $G$ with $m =m(G)$. Since $G$ is connected, $n\geq 4$ (because $n$ is even) and $m(G)=\frac{n}{2}$, there exist at least two edges in $M$, say $e_1$ and $e_2$, and a vertex $x\in e_1$ such that $N_{G}(x)\cap e_2\neq \emptyset$. Now, choose a maximal independent set $A'$ from $e_3\bigcup e_4\bigcup\cdots\bigcup e_m$ such that $A'\cap N_{G}(x)=\emptyset$. Then it is easy to verify that $A=A'\cup\{x\}$ is an independent set of $G$ and $N_{G}(A)$ is a vertex cover of $G$. Thus, by the combinatorial definition of the $\v$-number, we get $\v(G)\leq \vert A\vert =1+\vert A'\vert\leq 1+(m-2)=m-1=\lfloor\frac{n-1}{2}\rfloor$.
    \end{proof}

\begin{remark}
    From the additivity of $\v$-number (by \cite[Proposition 3.8]{jv21} or \cite[Proposition 3.9]{ss22}) and edge domination number (by definition) over the disjoint union of graphs, we have $\v(G)\leq \gamma_{e}(G)$ for any graph $G$. Also, if $G$ is a graph such that no connected component of $G$ is a $K_2$, then $\v(G)\leq \lfloor\frac{n-1}{2}\rfloor$.
\end{remark}
\begin{lemma}\label{main}
     Let $G$ be a graph with $x,y,z\in V(G)$ such that $\{x,y,z\}$ induces a $P_3$ with $\{x,y\},\{y,z\}\in E(G)$, $\deg_G (y)= 2$ and $\deg_G (z)=1$. If $A$ is an independent set such that $N_G(A)$ is a vertex cover of $G$, then $A\bigcap\{x,y,z\}\neq \emptyset$.
\end{lemma}
\begin{proof}
Given that $f=\{y,z\}$ is an edge of $G$. Since $N_G(A)$ is a vertex cover of $G$, we have $f\bigcap N_G(A)\neq\emptyset$.  If $y\in N_G(A)$, then $N_{G}(y)\bigcap A\neq\emptyset$. By the given conditions $\{x,y\},\{y,z\}\in E(G)$ and $\deg_G (y)= 2$, it follows that $N_{G}(y)=\{x,z\}$, and thus, $\{x,z\}\bigcap A\neq \emptyset$ in this case. Now, if $y\notin N_G(A)$, then we have $z\in N_G(A)$. Since $N_{G}(z)=\{y\}$ by the given conditions, we should have $y\in A$. This completes the proof.
\end{proof}

Next, we are going to prove our one of the main results, which gives an approximate range of the set $\RV(n)$.

\begin{theorem}\label{thm:RV(n)}
    Let $A(n)=\{(r,v)\in \mathbb{N}^{2} \mid 1\leq r<\frac{n}{2}, 1\leq v\leq r-\lceil\frac{r}{n-2r}\rceil+1\}$ and $B(n)=\{(r,v)\in\mathbb{N}^2 \mid 1\leq r<\frac{n}{2}, 1\leq v<\frac{n}{2}\}$. Then for $n\geq 3$, we have
    $$A(n)\subseteq \mathcal{RV}(n)\subseteq B(n).$$
\end{theorem}
\begin{proof}
    Let $G$ be a connected graph on $n\geq 3$ vertices. Then we have $\reg(G)<\frac{n}{2}$ by \Cref{prop:reg<n/2}. Again, due to \Cref{prop:v<=mat}, one can see that $\v(G)<\frac{n}{2}$. Thus, we have $\mathcal{RV}(n)\subseteq B(n)$.\par 

    Now, let us take $(r,v)\in A(n)$. We have to show that there exists a connected graph $G$ on $n$ vertices with $\reg(G)=r$ and $\v(G)=v$. Consider the following cases depending on the values of $n$ and $r$:
    \\
    \noindent \textbf{Case-I.} If $n-2r\geq r$, then $r-\left\lceil\frac{r}{n-2r}\right\rceil+1=r$. This gives 
    $$\min\left\{r,~r-\left\lceil{\frac{r}{n-2r}}\right\rceil+1\right\}=r.$$
    Let $G$ be a graph with the vertex set
    $$V(G)=\{y_1,y_2,\ldots,y_r,z_1,z_2,\ldots,z_r\}\bigcup\{x_{1},\ldots,x_v\}\bigcup\{w_1,\ldots,w_{n-2r-v}\}.$$
    Since $(r,v)\in A(n)$, we have $v\leq r$. Then we can consider the edge set as
    \begin{align*}
        E(G)= &\{\{x_1,x_i\}~|~i=2,\ldots,v\}\bigcup \{\{x_{1},y_i\}~|~i=1,\ldots,r-v+1\}\\ & \hspace{1cm} \bigcup \{\{x_{i},y_{r-v+i}\}~|~i=2,\ldots,v\} \bigcup \{\{y_i,z_i\}~|~i=1,\ldots,r\} \\ & \hspace{3cm} \bigcup \{\{x_{1},w_i\}~|~i=1,\ldots,n-2r-v\}.
    \end{align*}
    \begin{figure}[h]
    \centering
    \begin{tikzpicture}
        [scale=.55]
        \draw [fill, color=blue] (-12,4) circle [radius=0.1];  
        \node at (-12.4,4) {$y_1$};
        \draw [fill, color=blue] (-12,2) circle [radius=0.1]; 
         \node at (-12.4,2) {$z_1$};
        \draw [fill, color=blue] (-10,4) circle [radius=0.1];
         \node at (-10.4,4) {$y_2$};
        \draw [fill, color=blue] (-10,2) circle [radius=0.1];
         \node at (-10.4,2) {$z_2$};
        \node at (-8.5,3) {$\cdots$};
        \draw [fill, color=blue] (-7,4) circle [radius=0.1]; 
         \node at (-7.7,4) {$y_{r-v+1}$};
        \draw [fill, color=blue] (-7,2) circle [radius=0.1]; 
        \node at (-7.7,2) {$z_{r-v+1}$};
        \draw (-7,4)--(-7,2);
        \draw (-10,4)--(-10,2);
        \draw (-12,4)--(-12,2);
        \draw [fill, color=blue] (-5,4) circle [radius=0.1];
         \node at (-5.9,4) {$y_{r-v+2}$};
         \node at (-5.9,2) {$z_{r-v+2}$};
        \draw [fill, color=blue] (-5,2) circle [radius=0.1]; 
        \draw (-5,4)--(-5,2);
        \node at (-3.5,3) {$\cdots$};
        \draw [fill, color=blue] (-2,4) circle [radius=0.1]; 
        \draw [fill, color=blue] (-2,2) circle [radius=0.1]; 
        \node at (-2.5,4) {$y_r$};
         \node at (-2.5,2) {$z_r$};
        \draw (-2,4)--(-2,2);
       \draw [fill, color=green] (-2,7) circle [radius=0.1];
       \node at (-2,7.3) {$x_v$};
       \draw (-2,7)--(-2,4);
       \node at (-3.5,7) {$\cdots$};
    \draw [fill, color=green] (-5,7) circle [radius=0.1];
    \node at (-5,7.3) {$x_2$};
    \draw [fill, color=green] (-7,7) circle [radius=0.1];
    \node at (-7.5,7.2) {$x_1$};
    
    \draw (-5,7)--(-5,4);
       \draw (-7,7)--(-7,4);
       \draw (-2,7) to[bend right=40] (-7,7);
       \draw (-7,7)--(-12,4);
       \draw (-5,7) to[bend right=40] (-7,7);
       \draw (-7,7)--(-10,4);
       \draw [fill, color=black] (-5.5,9) circle [radius=0.1];
       \node at (-4.7,9.3) {$w_{n-2r-v}$};
       \node at (-7,9) {$\cdots$};
       \draw [fill, color=black] (-8.5,9) circle [radius=0.1];
       \node at (-8.5,9.3) {$w_1$};
       \draw (-5.5,9)--(-7,7);
       \draw (-8.5,9)--(-7,7);
    \end{tikzpicture}
    \caption{A tree $G$ with $n-2r\geq r$ and $(\reg(G),\v(G))=(r,v)$}
    \label{fig:tree:v<r}
\end{figure}
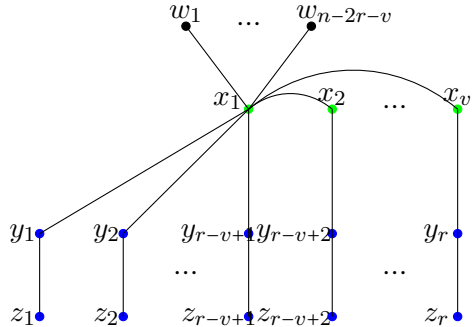
    Note that $G$ is a tree on $n$ vertices (see \Cref{fig:tree:v<r}) as there is no induced cycle in $G$. It is clear that $\{\{y_i,z_i\}~|~i=1,\ldots,r\}$ is an induced matching of $G$. Thus, by \Cref{thm:reg lower}, we have $\reg(G)\geq\im(G)\geq r$. Now, from the structure of the graphs $G-x_1$ and $G-N_{G}[x_1]$, one can easily obtain that $\im(G-x_1)=r$ and $\im(G-N_{G}[x_1])=v-1$. Since both $G-x_1$ and $G-N_{G}[x_1]$ are forests, $\reg(G-x_1)=r$ and $\reg(G-N_{G}[x_1])=v-1$ by \Cref{thm:reg forest}. Therefore, using \Cref{Dao}, we get  $\reg(G)\leq\max\{r,v-1+1\}=r$. Hence, $\reg(G)=r$. Next, we observe that $\{x_1,y_{r-v+2},\ldots,y_r\}$ is an independent set of $G$ whose neighbor set is a vertex cover of $G$. Then by \Cref{def:v}, we have $\mathrm{v}(G)\leq |\{x_1,y_{r-v+2},\ldots,y_r\}|= v$. For the reverse inequality, let $A$ be an independent set corresponding to the $\mathrm{v}$-number of $G$. Then $A$ is an independent set of $G$ such that $N_{G}(A)$ is a vertex cover of $G$ and $|A|=\v(G)$. From the structure of $G$, we see for each $i=1,\ldots,v$ that $B_i=\{x_{i},y_{r-v+i},z_{r-v+i}\}$ induces a $P_3$ as a subgraph of $G$ with $\deg_{G}(y_{r-v+i})=2$ and $\deg_{G}(z_{r-v+i})=1$. Therefore, by \Cref{main}, we have $B_i\bigcap A\neq \emptyset$ for each $i=1,\ldots,v$. Since $B_i\bigcap B_j=\emptyset$ for $i\neq j$, it follows that $|A|\geq v$. Hence, $\mathrm{v}(G)\geq v$ and consequently, $\v(G)=v$.\par 
    
     \noindent \textbf{Case-II.} Let $n-2r<r$. Then we have
    $$\min\left\{r,~r-\left\lceil{\frac{r}{n-2r}}\right\rceil+1\right\}=r-\left\lceil{\frac{r}{n-2r}}\right\rceil+1.$$ 
    Let $v=r-\left\lceil{\frac{r}{n-2r}}\right\rceil+1-p$ for some $0\leq p\leq r-\left\lceil{\frac{r}{n-2r}}\right\rceil$. We have to show that there exists a connected graph $G$ on $n$ vertices with $\reg(G)=r$ and $\v(G)=v$. Let us take $G$ to be a graph with the vertex set $$V(G)=\bigcup_{i=1}^{n-2r}\{y_{i1},y_{i2},\ldots,y_{it_i},z_{i1},z_{i2},\ldots,z_{it_i}\}\bigcup \{x_i~|~i=1,\ldots,n-2r\},$$
    where some $t_i$ may be $0$, and by $t_i=0$, we mean no such $y_{ij}$ and $z_{ij}$ exists for that particular $i$. Now, we consider the following edge set 
    \begin{align*}
        E(G) & =\bigcup_{i=1}^{n-2r}\bigcup_{j=1}^{t_i}\{\{x_i,y_{ij}\}, \{y_{ij},z_{ij}\}\}\bigcup \{\{x_1,x_i\}\mid i\in [n-2r] \text{ and }t_i=0\}\\
        &\hspace{4cm}\bigcup\{\{x_i,x_j\}~|~ i,j\in [n-2r] \text{ and } t_i\neq 0, t_j\neq 0\},
    \end{align*}
    where $\sum_{i=1}^{n-2r}t_i=r$ with $t_1=\lceil\frac{r}{n-2r}\rceil+p$ and $0\leq t_i\leq t_1$ for each $i=2,\ldots,n-2r$. In \Cref{fig:rv2}, we illustrate the structure of the graph $G$ defined for this case, where $t_{i}\neq 0$ for all $i\in [q]$ and $t_i=0$ for all $i\in\{q+1,\ldots,n-2r\}$.
    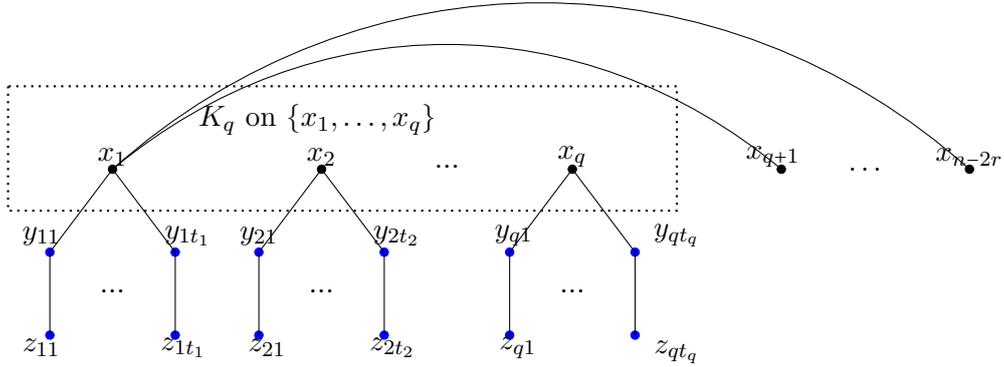
\begin{figure}[h]
    \centering
    \begin{tikzpicture}
        [scale=.55]
    \draw [fill, color=blue] (1,2) circle [radius=0.1];  
    \node at (0.8,2.4) {$y_{11}$};
    \draw [fill, color=blue] (1,0) circle [radius=0.1]; 
    \node at (0.8,-0.3) {$z_{11}$};
     
    \draw (1,0)--(1,2);
    \draw [fill, color=black] (2.5,4) circle [radius=0.1];   \node at (2.5,4.3) {$x_1$};
     \draw [fill, color=blue] (4,2) circle [radius=0.1];   
    \draw [fill, color=blue] (4,0) circle [radius=0.1];
    \node at (4.2,-0.3) {$z_{1t_1}$};
     \node at (4.3,2.4) {$y_{1t_1}$};
     \draw (4,0)--(4,2); 
     \draw(2.5,4)--(4,2);
     \draw(2.5,4)--(1,2);
     \node at (2.5,1) {$\cdots$};
     
     \draw [fill, color=blue] (6,2) circle [radius=0.1];   
    \draw [fill, color=blue] (6,0) circle [radius=0.1]; 
    \node at (6.2,-0.3) {$z_{21}$};
     \node at (6.0,2.4) {$y_{21}$};
    \draw (6,0)--(6,2);
    \draw [fill, color=black] (7.5,4) circle [radius=0.1];   
     \draw [fill, color=blue] (9,2) circle [radius=0.1];   
    \draw [fill, color=blue] (9,0) circle [radius=0.1];
     \node at (9.2,-0.3) {$z_{2t_2}$};
     \node at (9.3,2.4) {$y_{2 t_2}$};
     \draw (9,0)--(9,2); 
     \draw(7.5,4)--(6,2);
     \draw(7.5,4)--(9,2);
     \node at (7.5,4.3) {$x_2$};
     \node at (7.5,1) {$\cdots$};
     \node at (10.5,4) {\textbf{$\cdots$}};
      \draw [fill, color=blue] (12,2) circle [radius=0.1];   
    \draw [fill, color=blue] (12,0) circle [radius=0.1];
    \node at (12.2,-0.3) {$z_{q1}$};
     \node at (12.1,2.4) {$y_{q1}$};
    \draw (12,2)--(12,0);
    \draw [fill, color=black] (13.5,4) circle [radius=0.1];
    \node at (13.5,4.3) {$x_{q}$};
    \draw [fill, color=blue] (15,0) circle [radius=0.1];
    \draw [fill, color=blue] (15,2) circle [radius=0.1];
    \node at (16.0,-0.4) {$z_{qt_{q}}$};
     \node at (16.0,2.4) {$y_{qt_{q}}$};
    \draw(15,0)--(15,2);
    \draw (15,2)--(13.5,4);
    \draw (12,2)--(13.5,4);
     \node at (13.5,1) {$\cdots$};
      \draw[dotted, thick] (0,6) rectangle (16,3);
       \node at (7.4,5.2)  {$K_{q}$ on $\{x_1,\ldots,x_{q}\}$};
    \node at (18.3,4.3) {$x_{q+1}$};
     \draw [fill] (18.5,4.0) circle [radius=0.1]; 
     \draw (18.5,4.0) to[bend right=40] (2.5,4);
     \node at (20.5,4.0) {$\ldots$};
     \draw [fill] (23,4.0) circle [radius=0.1]; 
     \node at (23,4.3) {$x_{n-2r}$};
      \draw (23,4.0) to[bend right=42] (2.5,4);
    \end{tikzpicture}
    \caption{A chordal graph $G$ with $n-2r<r$ and $(\reg(G),\v(G))=(r,v)$}
    \label{fig:rv2}
\end{figure}
    Then observe that $G$ is a connected chordal graph on $n$ vertices and $t_1\geq 2$. Again, for each $0\leq p\leq r-\left\lceil{\frac{r}{n-2r}}\right\rceil$, the existence of such $G$ is guaranteed due to the fact that $(n-2r)\left\lceil{\frac{r}{n-2r}}\right\rceil\geq r$. Note that the set of edges $M=\bigcup_{i=1}^{n-2r}\{\{y_{ij},z_{ij}\}~|~1\leq j\leq t_i\}$ is clearly a maximal induced matching of the graph $G$. Therefore, we have
    $$\im(G)\geq |M|=\sum_{i=1}^{n-2r}t_i=r.$$
    For the reverse inequality, let $M'$ be an induced matching of $G$ with maximum cardinality, i.e., $|M'|=\im(G)$. If $M'$ does not contain any edge $e$ such that $x_i\in e$ for some $1\leq i\leq n-2r$, then the only choice of $M'$ is $M$. In this case, $\im(G)=r$. Now, suppose $M'$ contains an edge $e$ such that $x_i\in e$ for some $i\in\{1,\ldots,n-2r\}$. If $t_i=0$, then the only possibility is $e=\{x_1,x_i\}$ and no other edge of $M'$ contains any $x_k$ by our choice of $E(G)$. In this case, we should have $\bigcup_{i=2}^{n-2r}\{\{y_{ij},z_{ij}\}~|~1\leq j\leq t_i\}\subseteq
    M'$ and $\{y_{1j},z_{1j}\}\notin M'$ for all $1\leq j\leq t_1$. Then $M''=(M'\setminus \{e\})\bigcup \{\{y_{1j},z_{1j}\}\mid 1\leq j\leq t_1\}$ is an induced matching of $G$. Since $t_1\geq 2$, we have $|M''|>|M'|=\im(G)$, which is not possible. Therefore, we can assume $t_i\neq 0$. Then either $e=\{x_i,x_j\}$ for some $j\in [n-2r]\setminus\{i\}$ or $e=\{x_i,y_{ij}\}$ for some $j\in [t_i]$. In this case, we also see that $M'$ does not contain any other edge containing any $x_k$, because the induced subgraph of $G$ on the vertex set $\bigcup_{i=1}^{n-2r}\{x_i,y_{ij}\mid 1\leq j\leq t_i\}$ has induced matching number $1$ (specifically, it is a complete graph with some whiskers attached). Thus, considering the induced matching $M''=(M'\setminus \{e\})\cup \{\{y_{ij},z_{ij}\}\mid 1\leq j\leq t_i\}$, we can see that $M''\subseteq M$ as $M$ is a maximal induced matching containing $\bigcup_{i=1}^{n-2r}\{\{y_{ij},z_{ij}\}~|~1\leq j\leq t_i\}$. Thus, we have $r=|M|\geq |M''|\geq |M'|=\im(G)$ as $t_i\geq 1$. Hence, we finally get $\im(G)=r$. Since $G$ is a chordal graph, $\reg(G)=\im(G)=r$ by \Cref{thm:reg chordal}. Next, our aim is to show that $\v(G)=v$. Let us take the independent set $B=(\bigcup_{i=2}^{n-2r}\{y_{ij}~|~1\leq j\leq t_i\})\bigcup\{x_1\}$ of $G$. Then, one can easily verify that $N_G(B)$ is a vertex cover of $G$. Therefore, by \Cref{def:v}, we have $$\mathrm{v}(G)\leq |B|=\sum_{i=2}^{n-2r}t_i+1=r-\left\lceil\frac{r}{n-2r}\right\rceil-p+1.$$ 
    For the reverse inequality, let $A$ be an independent set of $G$ corresponding to $\mathrm{v}(G)$, i.e., $A$ is an independent set of $G$ with $N_{G}(A)$ as a vertex cover of $G$ and $|A|=\v(G)$. Since $\{x_1,\ldots,x_{n-2r}\}$ induces a complete graph with a possibility of some whiskers attached at $x_1$, the set $A$ contains at most one $x_i$. If possible, let $x_i\notin A$ for all $i=1,\ldots,n-2r$. For each $i\in\{1,\ldots,n-2r\}$ with $t_i\neq 0$, consider the collection of sets $B_{ij}=\{x_i,y_{ij},z_{ij}\}$, where $j=1,\ldots,t_i$. Note that $B_{ij}$ induces a $P_3$ as a subgraph of $G$ with $\deg_{G}(y_{ij})=2$ and $\deg_{G}(z_{ij})=1$. Then it follows from \Cref{main} that either $y_{ij}\in A$ or $z_{ij}\in A$ as $x_i\notin A$. This gives $\mathrm{v}(G)=|A|\geq\sum_{i=1}^{n-2r}t_i=r$, which is not possible as we already have $\mathrm{v}(G)\leq r-\lceil \frac{r}{n-2r}\rceil-p+1$ and $\lceil\frac{r}{n-2r}\rceil\geq 2$. Therefore, $A$ must contain one $x_i$ for some $i\in\{1,\ldots,n-2r\}$, and moreover using the previous argument, we can say $t_i\neq 0$.  Consider $x_m\in A$ for some $m\in\{1,\ldots,n-2r\}$ with $t_m\neq 0$. Since $B_{ij}\bigcap A\neq\emptyset$ due to the \Cref{main} and $(\{x_1,\ldots,x_{n-2r}\}\setminus\{x_m\})\cap A=\emptyset$, we should have $\{y_{ij},z_{ij}\}\bigcap A\neq \emptyset$ for each $i\in\{1,\ldots,n-2r\}\setminus\{m\}$ with $t_i\neq 0$, where $j\in\{1,\ldots,t_i\}$ for each $i$. This gives 
    $$\mathrm{v}(G)=|A|\geq 1+\sum_{i=1}^{n-2r}t_i-t_m\geq r-t_1+1=r-\left\lceil\frac{r}{n-2r}\right\rceil-p+1.$$ 
    Hence, $\mathrm{v}(G)=r-\lceil\frac{r}{n-2r}\rceil-p+1$ and this completes the proof.
\end{proof}

\section{The set $\RV_{W}(n)$: Restriction of $\RV(n)$ to whisker graphs}\label{sec:whisker graph} 

In this section, we explicitly find all the lattice points arising as a pair $(\reg(G),\v(G))$, where $G$ ranges over all connected whisker graph on a fixed number of vertices. In particular, we find the subset $\RV_{W}(n)$ of $\RV(n)$. We begin with the definition of whisker graph and some known results.

\begin{definition}
    Let $G$ be a graph with $V(G)=\{x_1,\ldots,x_m\}$. Consider the graph $W_G$ by adding a new set of vertices $Y=\{y_1,\ldots,y_m\}$ to $G$ and attaching the edges $\{x_i,y_i\}$ to $G$ for each $1\leq i\leq m $. The graph $W_G$ is known as the \textit{whisker} graph of $G$, and the attached edges $\{x_i,y_i\}$ are called the \textit{whiskers}. Note that a whisker graph always have even number of vertices.
\end{definition}

A whisker graph is a very well-covered graph satisfying the conditions of \cite[Lemma 4.1]{mmcrty11}. Thus, by \cite[Theorem 4.12]{mmcrty11}, we have the following.
\begin{theorem}\label{thm:reg whisker}
   Let $G$ be a simple connected graph and $W_G$ be the whisker graph of $G$. Then we have $\reg(W_G)= \im(W_G)$. Moreover, one can see that $\im(W_G)=\alpha(G)$.
\end{theorem}

\begin{theorem}\cite[Theorem 4.12]{ss22}\label{thm:whisker vim}
    Let $G$ be a simple connected graph and $W_G$ be the whisker graph of $G$. Then $\mathrm{v}(W_G)\leq \im(W_G)$.
\end{theorem} 

\begin{lemma}\label{lem:wrv}
    Let $G$ be a connected graph with $m\geq 2$ vertices and $W_G$ be its whisker graph. Then the following hold true.
\begin{enumerate}
    \item[(a)] $1\leq\reg(W_G)\leq m-1,$
    \item[(b)] $1\leq \mathrm{v}(W_G)\leq r-\lceil\frac{r}{m-r}\rceil+1$, where $r=\reg(W_G)$.
\end{enumerate}
\end{lemma}
\begin{proof}
    (a) Since $W_G$ has $2m$ vertices, it follows from \Cref{prop:reg<n/2} that $\reg(W_G) \leq m-1$. Also, $\reg(W_G)\geq 1$ is obvious.\par 
    
    (b) Let $G$ be a connected graph with $|V(G)|=m$ and $W_G$ be its whisker graph. Then $\reg(W_G)=\im(W_G)=\alpha(G)=r$ by \Cref{thm:reg whisker}. If $r\leq m-r,$ then $r-\lceil\frac{r}{m-r}\rceil+1=r$, and we have $\v(W_G)\leq r$ by \Cref{thm:whisker vim}. Now, let us consider the case $r>m-r$. Then by division algorithm, we can write $r=(m-r)k+s$ for some $k\geq 1$ and $0\leq s<m-r$. Let $B$ be a maximal independent set of $G$ with cardinality $r$ and $A=V(G)\setminus B$. Consider a set $C=\{v_1,\ldots,v_p\}\subseteq A$, which minimally covers $B$, i.e., $B\subseteq N_{G}(C)$ and $C$ is minimal with this property. It is clear that $|C|=p\leq |A|= m-r$. Suppose $v_i$ covers $t_i$ many vertices of $B$ (say $B_i$) and $C_i=(N_A(B_i)\setminus\bigcup_{j\neq i}N_A(B_j))\setminus\{v_1,\ldots,v_p\}$ for $i=1,\ldots,p$. Set $|C_i|=q_i$ and note that $\sum_{i=1}^{p}q_i\leq m-r-p$. Now, we consider two possible cases:\\
    \noindent\textbf{Case-I.} Suppose $s>0$. If possible, let $t_i\leq k+q_i$ for each $i=1,\ldots,p$. Since each $v_i$ covers $t_i$ many vertices of $B$, $C$ covers at most $\sum_{i=1}^{p}t_i$ vertices of $B$. Now, we have 
    \begin{align*}
        \sum_{i=1}^{p}t_i &\leq kp+\sum_{i=1}^{p}q_i\\
         &\leq kp+m-r-p\\
         &=m-r+(k-1)p\\
         &\leq (m-r)+(k-1)(m-r)\quad (\text{since } p\leq m-r)\\
         &= (m-r)k<r\quad (\text{as } r=(m-r)k+s \text{ and } s>0).
    \end{align*}
    This shows that $C$ covers at most $r-1$ many vertices of $B$, which contradicts the fact that $C$ minimally covers $B$. Therefore, there exists a $v_i\in C$ such that $t_i\geq (k+1)+q_i$. Consider the set $D=\{v_i\}\bigcup\{x\in C_i~|~x\notin N_A(v_i)\}\bigcup (B\setminus B_i)$. Note that $D$ is a maximal independent set of $G$ and 
    \begin{align*}
        |D|&\leq 1+|C_i|+|B\setminus B_i|\\
        &=1+q_i+r-t_i\\
        & \leq 1+q_i+r-(k+1)-q_i\\
        &=1+r-(k+1)\\
        &=r-\left\lceil\frac{r}{m-r}\right\rceil+1.
    \end{align*}
    The last equality is due to the fact that $r=(m-r)k+s$ with $s>0$. It is clear that $D$ is an independent set of $W_G$ and $V(G)\setminus D\subseteq N_{W_G}(D)$ as $D$ is a maximal independent set of $G$. Again, $N_{W_G}(D)$ contains the pendant vertices attached to the vertices in $D$. Thus, by the construction of $W_G$, we can see that $N_{W_G}(D)$ is a minimal vertex cover of $W_G$. Therefore, $\mathrm{v}(W_G)\leq |D|\leq r-\lceil\frac{r}{m-r}\rceil+1$ by \Cref{def:v}.\\
   \noindent \textbf{Case-II.} Suppose $s=0$. If possible, let $t_i<k+q_i$ for all $1\leq i\leq p$. It is clear that $C$ covers at most $\sum_{i=1}^{p}t_i$ vertices in $B$. Now, we have 
\begin{align*}
            \sum_{i=1}^{p}t_i & <kp+\sum_{i=1}^{p}q_i\\
            &\leq kp+(m-r-p)\\
            &=(m-r)+(k-1)p\\
            &\leq (m-r)k \quad (\text{since } p\leq m-r)\\
            &=r \quad (\text{as } r=(m-r)k+s \text{ and } s=0)
\end{align*}
Hence, $C$ covers at most $r-1$ vertices in $B$, which is a contradiction to the fact that $C$ covers $B$. Therefore, there exists a $v_i\in C$ such that $t_i\geq k+q_i$ for some $i$. Similar to the previous case, we consider $D=\{v_i\}\bigcup\{x\in C_i~|~x\notin N_A(v_i)\}\bigcup B\setminus B_i$. Here $D$ is a maximal independent set of $G$, and 
\begin{align*}
    |D| & \leq 1+|C_i|+|B\setminus B_i|\\
     & =1+q_i+r-t_i\\
     & \leq 1+q_i+r-k-q_i\\
     & =1+r-k\\
     & =r-\left\lceil\frac{r}{m-r}\right\rceil+1\quad (\text{since } r=(m-r)k).
\end{align*}
As mentioned in the first case, $D$ is an independent set of $W_G$ and $N_{W_G}(D)$ is a minimal vertex cover of $W_G$. Hence, by the combinatorial definition of the $\v$-number, we have $\mathrm{v}(W_G)\leq |D|\leq r-\lceil\frac{r}{m-r}\rceil+1 $. Finally, since $I(W_G)$ is not a prime ideal, we also have $\v(W_G)\geq 1$. This completes the proof. 
\end{proof}

The following is our main theorem of this section, which shows that the bounds of regularity and $\v$-number given in \Cref{lem:wrv} can be realized by connected whisker graphs. 

\begin{theorem}\label{thm:wrv}
For the class of connected whisker graphs $W_G$ on $n=2m$ vertices, the lattice points of $(\reg(W_G),\v(W_G))$ are given by
$$\mathcal{RV}_{W}(n)=\{(r,v)\mid 1\leq r\leq m-1 \text{ and } 1\leq v\leq r-\left\lceil\frac{r}{m-r}\right\rceil+1\}.$$
Moreover, since a whisker graph always has even number of vertices, we have $\RV_{W}(n)=\emptyset$ if $n$ is odd.
\end{theorem}

\begin{proof}
Let $W_G$ be a connected whisker graph on $n$ vertices. Then by \Cref{lem:wrv}, we have $1\leq r\leq m-1$ and $1\leq v\leq r-\left\lceil\frac{r}{m-r}\right\rceil+1$, where $m=\frac{n}{2}$. Therefore, we have 
$$\mathcal{RV}_{W}(n)\subseteq\{(r,v)\mid 1\leq r\leq m-1 \text{ and } 1\leq v\leq r-\left\lceil\frac{r}{m-r}\right\rceil+1\}.$$
For the reverse inclusion, let $r\leq m-1$ and $v\leq r-\lceil\frac{r}{m-r}\rceil+1$ be two positive integers. Then $v$ is of the form $r-\lceil\frac{r}{m-r}\rceil+1-p$ for some $0\leq p\leq r-\lceil\frac{r}{m-r}\rceil$. We have to show that there exists a connected graph $G$ on $m$ vertices such that $\reg(W_G)=r$ and $\v(W_G)=v$. Let us take $G$ to be a graph with the vertex set 
$$V(G)=\{x_1,\ldots,x_{m-r}\}\bigcup     \left(\bigcup_{i=1}^{k}\{y_{i1},\ldots,y_{it_i}\}\right),$$
and the edge set
\begin{align*}
    E(G) & = \left(\bigcup_{i=1}^{k}\{\{x_{i},y_{ij}\}\mid 1\leq j\leq t_i\}\right)\bigcup \{\{x_i,x_j\}\mid 1\leq i<j\leq m-r\}\\
    &\hspace{3cm} \bigcup \{\{y_{11},x_i\}\mid k+1\leq i\leq m-r\},
\end{align*}
where $k\leq m-r$, $\sum_{i=1}^{k}t_i=r$, $t_1=\lceil\frac{r}{m-r}\rceil+p$, and $1\leq t_i\leq t_1$ for each $i=2,\ldots,k$. Then $G$ is a connected graph on $m$ vertices. Note that for each $0\leq p\leq r-\left\lceil{\frac{r}{m-r}}\right\rceil$, the existence of such $G$ is guaranteed due to the fact that $(m-r)\left\lceil{\frac{r}{m-r}}\right\rceil\geq r$. We consider the whisker graph $W_G$ of $G$ with the vertex set and edge set as follows:
\begin{align*}
    V(W_G)& =V(G)\bigcup \{u_1,\ldots,u_{m-r}\}\bigcup \left(\bigcup_{i=1}^{k}\{z_{i1},\ldots,z_{it_i}\}\right),\\
    E(W_G)& =E(G)\bigcup \{\{x_{i},u_{i}\}\mid i\in [m-r]\}\bigcup \left(\bigcup_{i=1}^{m-r}\{\{y_{ij},z_{ij}\}\mid j\in [t_i]\}\right).
\end{align*}
    \begin{figure}[h]
    \centering
    \begin{tikzpicture}
        [scale=.55]
    \draw [fill, color=blue] (1,2) circle [radius=0.1];  
    \node at (0.8,2.4) {$y_{11}$};
    \draw [fill, color=blue] (1,0) circle [radius=0.1]; 
    \node at (0.8,-0.3) {$z_{11}$};
     
    \draw (1,0)--(1,2);
    \draw [fill, color=black] (2.5,4) circle [radius=0.1];  
    \node at (3,4.3) {$x_1$};
     \draw [fill, color=blue] (4,2) circle [radius=0.1];   
    \draw [fill, color=blue] (4,0) circle [radius=0.1];
    \node at (4.2,-0.3) {$z_{1t_1}$};
     \node at (4.3,2.4) {$y_{1t_1}$};
     \draw (4,0)--(4,2); 
     \draw(2.5,4)--(4,2);
     \draw(2.5,4)--(1,2);
     \node at (2.5,1) {$\cdots$};
     
     \draw [fill, color=blue] (6,2) circle [radius=0.1];   
    \draw [fill, color=blue] (6,0) circle [radius=0.1]; 
    \node at (6.2,-0.3) {$z_{21}$};
     \node at (6.0,2.4) {$y_{21}$};
    \draw (6,0)--(6,2);
    \draw [fill, color=black] (7.5,4) circle [radius=0.1];   
     \draw [fill, color=blue] (9,2) circle [radius=0.1];   
    \draw [fill, color=blue] (9,0) circle [radius=0.1];
     \node at (9.2,-0.3) {$z_{2t_2}$};
     \node at (9.3,2.4) {$y_{2t_2}$};
     \draw (9,0)--(9,2); 
     \draw(7.5,4)--(6,2);
     \draw(7.5,4)--(9,2);
     \node at (8,4.3) {$x_2$};
     \node at (7.5,1) {$\cdots$};
     \node at (10.5,4) {\textbf{$\cdots$}};
      \draw [fill, color=blue] (12,2) circle [radius=0.1];   
    \draw [fill, color=blue] (12,0) circle [radius=0.1];
    \node at (12.2,-0.3) {$z_{k1}$};
     \node at (12.1,2.4) {$y_{k1}$};
    \draw (12,2)--(12,0);
    \draw [fill, color=black
    ] (13.5,4) circle [radius=0.1];
    \node at (14,4.3) {$x_{k}$};
    \draw [fill, color=blue] (15,0) circle [radius=0.1];
    \draw [fill, color=blue] (15,2) circle [radius=0.1];
    \node at (15.8,-0.4) {$z_{kt_{k}}$};
     \node at (15.8,2.4) {$y_{kt_{k}}$};
    \draw (15,0)--(15,2);
    \draw (15,2)--(13.5,4);
    \draw (12,2)--(13.5,4);
     \node at (13.5,1) {$\cdots$};
      \draw[dotted, thick] (0,5.5) rectangle (26,3);
       \node at (13,5)  {$K_{m-r}$ on $\{x_1,\ldots,x_{m-r}\}$};
       \node at (14.2,7.3) {$u_{k}$};
        \node at (8,7.3) {$u_2$};
        \node at (3,7.3) {$u_1$};
        \draw [fill, color=black] (13.5,7) circle [radius=0.1];
        \draw [fill, color=black] (7.5,7) circle [radius=0.1];
        \draw [fill, color=black] (2.5,7) circle [radius=0.1];
        \draw (2.5,7)--(2.5,4);
         \draw (7.5,7)--(7.5,4);
         \draw (13.5,7)--(13.5,4);
         \draw [fill, color=black
    ] (18.5,4) circle [radius=0.1];
    \node at (19.3,4.3) {$x_{k+1}$};
    \draw [fill, color=black
    ] (18.5,7) circle [radius=0.1];
    \node at (19.3,7.3) {$u_{k+1}$};
    \draw (18.5,7)--(18.5,4);
     \node at (21,4) {$\ldots$};
      \draw [fill, color=black
    ] (23.5,4) circle [radius=0.1];
    \node at (24.4,4.3) {$x_{m-r}$};
    \draw [fill, color=black
    ] (23.5,7) circle [radius=0.1];
    \node at (24.4,7.3) {$u_{m-r}$};
    \draw (23.5,7)--(23.5,4);
    \draw(1,2)to[bend left=30](23.5,4);
    \draw(1,2)to[bend left=26](18.5,4);
    \end{tikzpicture}
    \caption{A whisker graph $W_G$ with $(\reg(W_G),\v(W_G))=(r,v)$}
    \label{fig:wrv}
\end{figure}
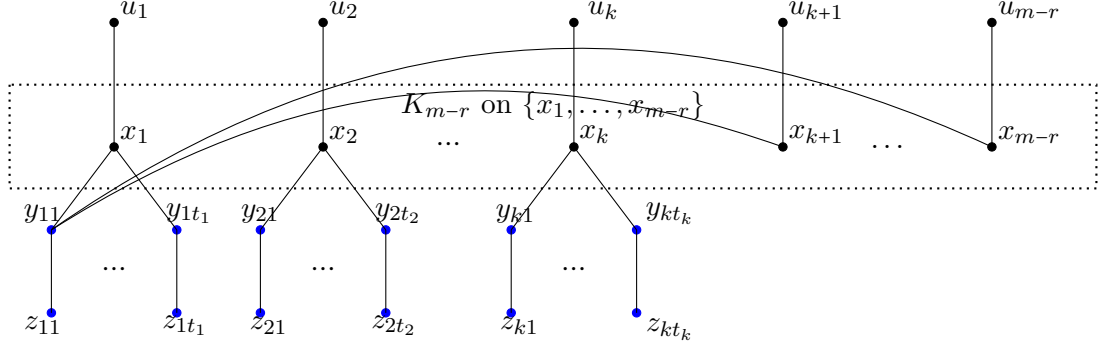
From the structure of $G$ (see \Cref{fig:wrv}), it is clear that the set of vertices $\bigcup_{i=1}^{k} \{y_{i1},\ldots,y_{i{t_i}}\}$ is a maximal independent set of $G$. Thus, the independence number $\alpha(G)\geq \sum_{i=1}^{k}t_i=r$. For the reverse inequality, let $D$ be a maximal independent set of $G$ with $\alpha(G)=|D|$. Since the induced subgraph of $G$ on the vertex set $\{x_1,\ldots,x_{m-r}\}$ forms a complete graph, $D$ can contain at most one $x_i$ among $\{x_1,\ldots,x_{m-r}\}$. If $D$ does not contain any $x_i$, then $D$ must be equal to $\bigcup_{i=1}^{k} \{y_{i1},\ldots,y_{i{t_i}}\}$, which gives $\alpha(G)=r$. Suppose $D$ contains $x_i$ for some $1\leq i\leq m-r$. If $i\in [k]$, consider $D'=(D\setminus \{x_i\})\bigcup \{y_{i1}\}$, otherwise, consider $D'=(D\setminus\{x_i\})\bigcup\{y_{11}\}$. Then, we have $|D|=|D'|$ and $D'\subseteq \bigcup_{i=1}^{k} \{y_{i1},\ldots,y_{i{t_i}}\}$. Therefore, $\alpha(G)=|D|\leq \sum_{i=1}^{k}t_i=r$. Consequently, we have $\alpha(G)=r$. Hence, by \Cref{thm:reg whisker}, $\reg(W_G)=\im(W_G)=\alpha(G)=r$. \par 

Next, our aim is to show that $\v(W_G)=v=r-\lceil\frac{r}{m-r}\rceil+1-p$. First, we consider the independent set $B=\bigcup_{i=2}^{k}\{y_{ij}\mid 1\leq j\leq t_i\}\bigcup\{x_1\}$ of $G$. Since $\{x_1,\ldots,x_{m-r}\}$ induces a complete graph in $W_G$, one can see that the neighbors $N_{W_G}(B)$ of $B$ forms a minimal vertex cover of $W_G$. Therefore, by \Cref{def:v}, we have
$$\v(W_G)\leq |B|=\sum_{i=2}^{k}t_i+1=r-\left\lceil\frac{r}{m-r}\right\rceil-p+1.$$
For the reverse inequality, let $A$ be an independent set of $W_G$ such that $(I(W_G):X_{A})=\la N_{W_G}(A)\ra$ and $|A|=\v(W_G)$. Since $\{x_1,\ldots,x_{m-r}\}$ induces a complete graph in $W_G$, the set $A$ can contains at most one $x_i$. For each $i\in\{2,\ldots,k\}$, consider the collection of sets $B_{ij}=\{x_i,y_{ij},z_{ij}\}$, where $j=1,\ldots,t_i$. Note that $B_{ij}$ induces a $P_3$ as a subgraph of $W_G$ with $\deg y_{ij}=2$ and $\deg z_{ij}=1$. Therefore, by \Cref{main}, we have $B_{ij}\bigcap A\neq \emptyset$. Again, we see that $\{y_{1j},z_{1j}\}\bigcap N_{W_G}(A)\neq \emptyset$ for each $j\in [t_1]$ as $N_{W_G}(A)$ is a minimal vertex cover of $W_G$. Thus, the set of vertices $N_{W_G}(y_{1j})\bigcup N_{W_G}(z_{1j})$ will intersect $A$ non-trivially for each $1\leq j\leq t_1$. Note that
\begin{align*}
   N_{W_G}(y_{1j})\bigcup N_{W_G}(z_{1j}) =
   \begin{cases}
       \{x_1,x_{k+1},\ldots,x_{m-r},y_{1j},z_{1j}\}\quad \text{if }j=1\\
       \{x_1,y_{1j},z_{1j}\}\quad \text{if } 2\leq j\leq t_1
   \end{cases}
\end{align*}
Since $A$ contains at most one $x_i$, combining all the above possibilities, it follows that there exists $s\in [k]$ such that
\begin{align*}
   \v(W_G) =|A| & \geq \sum_{i=1}^{k}t_i-t_s+1\\
   & \geq \sum_{i=1}^{k}t_i-t_1+1 \quad (\text{since }t_1\geq t_s)\\
   & =r-\left\lceil\frac{r}{m-r}\right\rceil-p+1.
\end{align*}
Hence, we finally have $\v(W_G)=r-\lceil\frac{r}{m-r}\rceil-p+1=v$ and this completes the proof.
\end{proof}

\section{The set $\RV_{CW}(n)$: Restriction of $\RV(n)$ to Cameron-Walker graphs}\label{sec:CW}

In this section, we identify all possible pairs of $(\reg(G),\v(G))$ when $G$ is restricted to Cameron-Walker graphs on a fixed number of vertices. In particular we find the following subset of $\RV(n)$:
$$\RV_{CW}(n)=\left\{ (r,v)\in \mathbb{N}^2 \;\middle|\; 
\begin{array}{l}
      \text{there exists a Cameron-Walker graph $G$ on $n$ vertices}\\
      \hspace{1cm}\text{ with } \reg(G)=r \text{ and } \v(G)=v
\end{array}
     \right\}.$$
Let us start with the definition and some properties of Cameron-Walker graphs.

A \textit{pendant edge} (or \textit{whisker}) of a graph is an edge incident to a vertex of degree $1$. Similarly, a \textit{pendant triangle} in a graph is a copy of $C_3$ (equivalently, $K_3$) having two vertices of degree $2$. In the literature, a pendant triangle is usually defined as an induced triangle in which two vertices have degree $2$ and the third vertex has degree greater than $2$. In our setting, however, the graph $C_3$ itself is also regarded as a pendant triangle. A graph is called a \textit{star graph} if it is obtained by attaching several pendant edges to a fixed vertex; in other words, a \textit{star graph} is a connected graph where, except for one vertex, all the other vertices have degree $1$. A graph is said to be a \textit{star triangle} if it is obtained by attaching one or more pendant triangles to a fixed vertex.
\par 

For a graph $G$, we always have $\im(G)\leq m(G)$ by definition. In 2005, Cameron-Walker classified all such connected graphs $G$ for which the equality $\im(G)=m(G)$ holds (see \cite[Theorem 1]{cw05}). A minor oversight in their classification was later corrected in \cite[Remark 0.1]{hhko15}. In particular, for a connected graph $G$, we have $\im(G)=m(G)$ if and only if $G$ is one of the following graphs:
\begin{enumerate}
    \item[(a)] a star graph;
    \item[(b)] a star triangle;
    \item[(c)] a connected finite graph consisting of a connected bipartite graph with vertex partition $\{v_1,\ldots,v_m\}\bigsqcup \{w_1,\ldots,w_m\}$ such that there is at least one whisker
attached to each vertex $v_i$ and that there may be possibly some pendant triangles attached to each vertex $w_j$ (see \Cref{fig:CW}).
\end{enumerate}

\begin{definition}
    A finite connected simple graph $G$ is called a \textit{Cameron-Walker} graph (named after Kathie Cameron and Tracy Walker) if $G$ is neither a star graph nor a star triangle and satisfies $\im(G)=m(G)$, i.e., $G$ is a graph of type (c) described above.
\end{definition}

\begin{figure}[htbp]
    \centering
   \begin{tikzpicture}
       [scale=0.55]
      \draw [fill, color=blue] (0,4) circle [radius=0.1]; 
      \node at (0,4.7) {$x_{1}^{(1)}$};
      \node at (1.5,3.5) {$\cdots$};
      \draw [fill] (1.5,2) circle [radius=0.1];
      \node at (1.5,1.6) {$u_{1}$};
       \draw [fill, color=blue] (3,4) circle [radius=0.1];
       \node at (3,4.7) {$x_{s_1}^{(1)}$};
       \draw [fill, color=blue] (5,4) circle [radius=0.1];
       \node at (5,4.7) {$x_{1}^{(2)}$};
        \node at (6.5,3.5) {$\cdots$};
         \draw [fill] (6.5,2) circle [radius=0.1]; 
          \node at (6.5,1.6) {$u_{2}$};
       \draw [fill, color=blue] (8,4) circle [radius=0.1];
        \node at (8,4.7) {$x_{s_2}^{(2)}$};
       \draw (0,4)--(1.5,2);
       \draw (3,4)--(1.5,2);
       \draw (5,4)--(6.5,2);
       \draw (8,4)--(6.5,2);
        \draw [fill, color=blue] (12,4) circle [radius=0.1];
        \node at (12,4.7) {$x_{1}^{(m)}$};
        \node at (13.5,1.6) {$u_m$};
         \draw [fill, color=blue] (15,4) circle [radius=0.1]; 
         \node at (15,4.7) {$x_{s_m}^{(m)}$};
          \draw [fill] (13.5,2) circle [radius=0.1]; 
         \draw (12,4)--(13.5,2);
          \draw (15,4)--(13.5,2);
          \node at (13.5,3.5) {$\cdots$};
          \node at (10,2.5) {$\cdots$};
          \draw[dotted, thick] (-1.5,-1) rectangle (16.5,2);
          \node at (7.4,0.5)  {Connected bipartite graph on $\{u_1,\ldots,u_m\}\cup\{w_1,\ldots,w_p\}$};
    \draw [fill, color=black] (1,-1) circle [radius=0.1];
    \node at (1,-0.6) {$w_1$};
    \node at (1,-2.5) {$\cdots$};
    \draw [fill, color=blue] (0,-3.5) circle [radius=0.1];
    \node at (0.1,-4.2) {$y_{12}^{(1)}$};
     \draw [fill, color=blue] (-0.5,-3) circle [radius=0.1];
     \node at (-1.5,-2.5) {$y_{11}^{(1)}$};
     \draw (0,-3.5)--(-0.5,-3);
     \draw (0,-3.5)--(1,-1);
     \draw (-0.5,-3)--(1,-1);
     \draw [fill, color=blue] (2,-3.5) circle [radius=0.1];
     \node at (2,-4.2) {$y_{t_11}^{(1)}$};
      \draw [fill, color=blue] (2.5,-3) circle [radius=0.1];
      \node at (3,-2.4) {$y_{t_12}^{(1)}$};
      \draw (2,-3.5)--(2.5,-3);
      \draw (2,-3.5)--(1,-1);
      \draw (2.5,-3)--(1,-1);
      \draw [fill, color=black] (6.5,-1) circle [radius=0.1];
      \node at (6.5,-0.6) {$w_2$};
      \draw [fill, color=blue] (5,-3) circle [radius=0.1];
      \node at (4.5,-2.5) {$y_{11}^{(2)}$};
      \draw [fill, color=blue] (5.5,-3.5) circle [radius=0.1];
      \node at (5.6,-4.2) {$y_{12}^{(2)}$};
      \draw [fill, color=blue] (8,-3) circle [radius=0.1];
      \node at (8.4,-2.5) {$y_{t_22}^{(2)}$};
      \draw [fill, color=blue] (7.5,-3.5) circle [radius=0.1];
      \node at (7.6,-4.2) {$y_{t_21}^{(2)}$};
      \draw (5,-3)--(5.5,-3.5);
      \draw (5,-3)--(6.5,-1);
      \draw (5.5,-3.5)--(6.5,-1);
      \draw (7.5,-3.5)--(6.5,-1);
      \draw (8,-3)--(6.5,-1);
      \draw (7.5,-3.5)--(8,-3);
      \node at (6.5,-2.5) {$\cdots$};
       \draw [fill, color=black] (14,-1) circle [radius=0.1];
       \node at (10.3,-2) {$\cdots$};
       \node at (14,-0.6) {$w_p$};
       \node at (14,-2.5) {$\cdots$};
      \draw [fill, color=blue] (13,-3.5) circle [radius=0.1];
      \draw [fill, color=blue] (12.5,-3) circle [radius=0.1];
      \draw [fill, color=blue] (15,-3.5) circle [radius=0.1];
      \node at (15.1,-4.2) {$y_{t_p1}^{(p)}$};
      \node at (13.1,-4.2) {$y_{12}^{(p)}$};
      \node at (12.1,-2.5) {$y_{11}^{(p)}$};
      \node at (15.9,-2.5) {$y_{t_p2}^{(p)}$};
      \draw [fill, color=blue] (15.5,-3) circle [radius=0.1];
      \draw (15.5,-3)--(15,-3.5);
      \draw (13,-3.5)--(12.5,-3);
      \draw (14,-1)--(15.5,-3);
      \draw (14,-1)--(15,-3.5);
      \draw (14,-1)--(13,-3.5);
      \draw (14,-1)--(12.5,-3);
   \end{tikzpicture}
    \caption{A Cameron-walker graph}
    \label{fig:CW}
\end{figure}
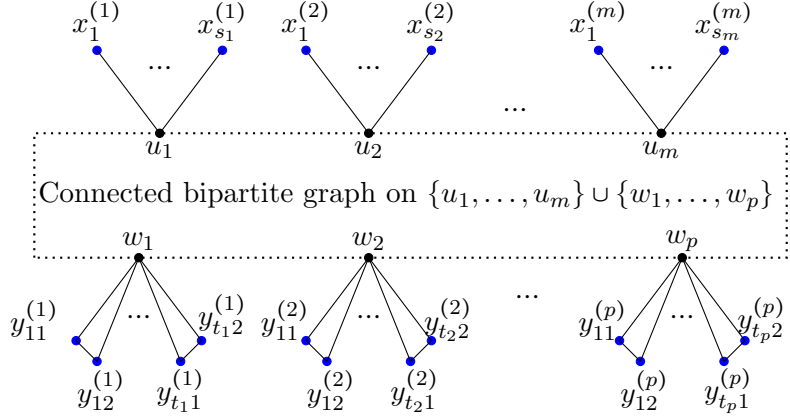

Note that for a Cameron-Walker graph with the notation in \Cref{fig:CW}, we always have $s_{i}\geq 1$ for each $i\in [m]$ and $t_{j}\geq 0$ for each $j\in [p]$. The following result will be used in the rest of this section.

\begin{theorem}\cite[Theorem 4.2]{hkmv22}\label{cameron regularity}
Let $G$ be a Cameron-Walker graph with the notation as in Figure \ref{fig:CW}. Then
\begin{enumerate}[(i)]
    \item $|V(G)|=m+p+\displaystyle\sum_{i=1}^{m}s_i+2\displaystyle\sum_{j=1}^{p}t_j;$
    \item $\reg(G)=m+\displaystyle\sum_{j=1}^{p}t_j$.
\end{enumerate}
    \end{theorem}

The following lemma establishes sharp upper bounds for the regularity and the $\v$-number of Cameron-Walker graphs.

\begin{lemma}\label{lem:cwrv}
    Let $G$ be a Cameron-Walker graph on $n$ vertices. Then we have $n\geq 5$ and
    \begin{enumerate}
        \item[(a)] $2\leq \reg(G)\leq \left\lfloor \frac{n-1}{2}\right\rfloor,$
        \item[(b)] $1\leq \mathrm{v}(G)\leq \min\{ r-1, n-2r\}$, where $r=\reg(G)$.
    \end{enumerate}
 \end{lemma}
 
\begin{proof}
  a) From \cite[Theorem 5.1]{hkmv22}, it follows that $2\leq \reg(G)\leq \lfloor\frac{n-1}{2}\rfloor$.\\
  \noindent b) Let $G$ be a Cameron-Walker graph with the notation as shown in \Cref{fig:CW}. Then observe that
\begin{align}\label{eq1}
   n & =m+p+\sum_{i=1}^{m}s_i+2\sum_{j=1}^{p}t_j \nonumber\\
   & \geq 2m+p+2\sum_{j=1}^{p}t_j,
\end{align}
due to the fact that $s_i\geq 1$ for all $1\leq i\leq m$. In this regard, one should note that some of $t_j$ may be $0$, that means there are no triangles attached to the vertex $w_j$. Now, by \Cref{cameron regularity}, we have $\reg(G)=r=m+\sum_{j=1}^{p}t_j$, which gives $n\geq 2r+p$ by \Cref{eq1}, equivalently, $n-2r\geq p$. From the structure of $G$ as shown in \Cref{fig:CW}, it is clear that $W=\{w_1,\ldots,w_p\}$ is an independent set of $G$ such that $N_G(W)$ is a vertex cover of $G$. Therefore, by the definition of the $\v$-number of a graph, we have $\mathrm{v}(G)\leq |W|=p$. Consequently, it follows that $\mathrm{v}(G)\leq n-2r$. Now, it remains to show that $\v(G)\leq r-1$. If $p<m$, then we have $\mathrm{v}(G)\leq p\leq m-1\leq r-1$. Suppose $p\leq m$. Then we consider two possible cases:\\
      \noindent \textbf{Case-I.} Let $t_j\neq 0$ for some $1\leq j\leq p$. Without loss of generality, we assume $t_{j}\neq 0$ for all $1\leq j\leq c$ and $t_j=0$ for all $c+1\leq j\leq p$ if $c<p$. Now, we consider the set of vertices $A=\{w_1,\ldots,w_c\}\bigcup (U\setminus\bigcup_{j=1}^{c}N_U(w_j))$, where $U=\{u_1,\ldots,u_m\}$. Then from the structure of $G$ (see \Cref{fig:CW}), it is clear that $A$ is an independent set of $G$ and $N_G(A)$ is a vertex cover of $G$. Therefore, by the definition of $\v$-number, we have
     \begin{align*}
        \v(G)\leq |A| & =c+m-|\bigcup_{j=1}^{c}N_U(w_j)|.
     \end{align*}
     Since $G$ is connected, $\bigcup_{j=1}^{c}N_{U}(w_j)\neq \emptyset$, i.e. $|\bigcup_{j=1}^{c}N_U(w_j)|\geq 1$. Thus,
     $$\v(G)\leq c+m-1\leq \sum_{j=1}^{p}t_j+m-1=r-1.$$
  \noindent \textbf{Case-II.} Let $t_j=0$ for all $1\leq j\leq p$. Then we can see that $m\geq 2$ as $G$ is not a star graph. Now, $G$ being connected there exists a path between $u_1$ and $u_2$. From the structure of $G$, it is clear that the path between $u_1$ and $u_2$ actually lie in the induced subgraph $G[S]$, where $S=U\cup W$ with $W=\{w_1,\ldots,w_p\}$. Since the graph $G[S]$ is a connected bipartite graph with the partite sets $U$ and $W$, there exists $u_i,u_j\in U$ with $i\neq j$ and $w_k\in W$ such that $\{u_{i},w_k\}, \{u_{j},w_k\}\in E(G[S])$. Now, let us take the set $A=\{w_k\}\bigcup (U\setminus N_{U}(w_k))$. Then one can easily verify that $A$ is an independent set of $G$ such that $N_{G}(A)$ is a vertex cover of $G$. Therefore, by the definition of $\v$-number, we have 
  \begin{align*}
      \v(G)& \leq |A|\\
      & = 1+|U|-|N_{U}(w_k)|\\
      & \leq 1+|U|-2 \quad (\text{since } |N_{U}(w_k)|\geq 2)\\
      & =m-1=r-1 \quad (\text{since } t_j=0 \text{ for all } j\in[p]).
  \end{align*}
  Hence, we finally get $\v(G)\leq \min\{r-1,n-2r\}$, where $r=\reg(G)$.
\end{proof}
The following theorem is the main result of this section. Note that \Cref{lem:cwrv} is in fact a part of this theorem and has been stated separately for the convenience of the reader.

\begin{theorem}\label{thm:cwrv}
    For the class of Cameron-Walker graphs, we have $\mathcal{RV}_{CW}(n)=\emptyset$ for $n<5$, and for $n\geq 5$, we have
     $$\mathcal{RV}_{CW}(n)=\left\{(r,v)\mid 2\leq r\leq \left\lceil\frac{n-1}{2}\right\rceil \text{ and } 1\leq v\leq \min\{r-1,n-2r\}\right\}.$$
\end{theorem}
\begin{proof}
    From the definition of Cameron-Walker graph it follows that there is no Cameron-Walker graph on $n<5$ vertices. So, we have $\mathcal{RV}_{CW}(n)=\emptyset$ for $n<5$. Now, we assume $n\geq 5$. Due to \Cref{lem:cwrv}, we have 
    $$\mathcal{RV}_{CW}(n)\subseteq\left\{(r,v)\mid 2\leq r\leq \left\lceil\frac{n-1}{2}\right\rceil \text{ and } 1\leq v\leq \min\{r-1,n-2r\}\right\}.$$
    for the reverse inclusion, let $r,v$ be two integers satisfying $2\leq r\leq \left\lceil\frac{n-1}{2}\right\rceil$ and $1\leq v\leq \min\{r-1,n-2r\}$. We have to show that there exists a Cameron-Walker graph $G$ on $n$ vertices such that $\reg(G)=r$ and $\v(G)=v$. Let $G$ be a Cameron-Walker graph on $n\geq 5$ vertices with the vertex set 
    $$V(G)=\{w_1,\ldots,w_{n-2r},u_1,\ldots,u_r,x_1,\ldots,x_r\}$$ and the edge set
    \begin{align*}
        E(G)&=\{\{w_1,u_i\}~|~i=1,\ldots,r-v+1\}\bigcup\{\{w_i,u_{r-v+i}\}~|~i=2,\ldots,v\}\\
        &\hspace{2cm} \bigcup\{\{w_i,u_{r-v+1}\}~|~i=2,\ldots,n-2r\}\bigcup\{\{u_i,x_i\}~|~i=1,\ldots,r\}.
    \end{align*}
    
   \begin{figure}[htbp]
    \centering
   \begin{tikzpicture}
       [scale=0.55]
      \draw [fill, color=blue] (0,2) circle [radius=0.1];
      \draw [fill, color=blue] (0,0) circle [radius=0.1];
      \node at (-0.4,0) {$u_1$};
      \node at (-0.4,2) {$x_1$};
      \draw (0,2)--(0,0);
      \draw (0,0)--(2,-1);
      \draw [fill, color=black] (2,-1) circle [radius=0.1];
      \node at (2,-1.4) {$w_1$};
      \draw [fill, color=blue] (4,2) circle [radius=0.1];
      \draw [fill, color=blue] (4,0) circle [radius=0.1];
      \node at (5,0.2) {$u_{r-v+1}$};
      \node at (5,2) {$x_{r-v+1}$};
      \draw [fill, color=black] (4,-1) circle [radius=0.1];
      \node at (4,-1.4) {$w_2$};
      \draw (4,2)--(4,0);
      \draw (4,0)--(2,-1);
      \draw (4,0)--(4,-1);
      \draw(4,-1)--(7,0);
      \draw(7,0)--(7,2);
      \draw (7,-1)--(4,0);
      \draw [fill, color=blue] (7,2) circle [radius=0.1];
      \draw [fill, color=blue] (7,0) circle [radius=0.1];
      \node at (8,0) {$u_{r-v+2}$};
      \node at (8,2) {$x_{r-v+2}$};
      \draw [fill, color=black] (7,-1) circle [radius=0.1];
      \node at (7,-1.4) {$w_3$};
      \draw [fill, color=blue] (10,2) circle [radius=0.1];
      \draw [fill, color=blue] (10,0) circle [radius=0.1];
      \node at (11,0) {$u_{r-v+3}$};
      \node at (11,2) {$x_{r-v+3}$};  
      \draw (10,0)--(10,2);
      \draw (7,-1)--(10,0);
       \draw [fill, color=blue] (14,2) circle [radius=0.1];
      \draw [fill, color=blue] (14,0) circle [radius=0.1];
      \node at (14.4,0) {$u_{r}$};
      \node at (14.4,2) {$x_{r}$}; 
       \draw [fill, color=black] (12,-1) circle [radius=0.1];
      \node at (12,-1.4) {$w_v$}; 
       \draw (4,0)--(12,-1); 
       \draw (14,0)--(12,-1); 
       \draw (14,0)--(14,2); 
        \draw [fill, color=black] (14.5,-1) circle [radius=0.1];
        \draw [fill, color=black] (17,-1) circle [radius=0.1];
        \node at (17.3,-1.4) {$w_{n-2r}$};
        \node at (15.1,-1.4) {$w_{v+1}$};
        \node at (15.2,-1) {$\cdots$};
        \draw (4,0)--(14.5,-1);
        \draw (4,0)--(17,-1);
        \node at (2,1) {$\cdots$};
        \node at (12,1) {$\cdots$};
        \end{tikzpicture}
    \caption{A Cameron-walker graph $G$ with $(\reg(G),\v(G))=(r,v)$}
    \label{fig:CW-exists}
\end{figure}
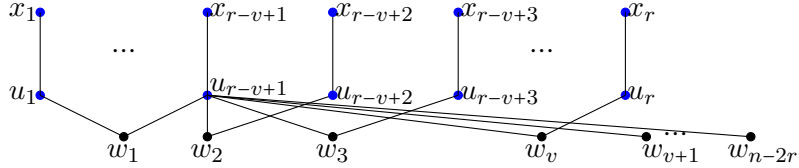
  \noindent The graph $G$ has been shown in \Cref{fig:CW-exists}. Compare the labeling of the graph $G$ with the labeling of the general Cameron-Walker graph as shown in \Cref{fig:CW}. From the structure of the graph $G$, we see that no triangles are there in $G$, and so, we have $t_j=0$ for all $1\leq j\leq n-2r$. Therefore, by \Cref{cameron regularity}, we get $\reg(G)=r$. Now, to show $\v(G)=v$, we first consider the the independent set $A=\{w_1,\ldots,w_v\}$. Since $\{u_1,\ldots,u_r\}\subseteq N_{G}(A)$ and $\{x_1,\ldots, x_r,w_{v+1},\ldots,w_{n-2r}\}$ is a set of pendant vertices (i.e., vertices of degree $1$), we see that $N_{G}(A)$ is a vertex cover of $G$. Thus, by the definition of $\v$-number, we get $\v(G)\leq |A|=v$. To prove the reverse inequality, let us consider an independent set $B$ corresponding to $\v(G)$, i.e., $N_{G}(B)$ is a minimal vertex cover of $G$ and $|B|=\v(G)$. Now, we consider the following sets of vertices: $B_1=\{w_1,u_1,x_1\}$ and $B_i=\{w_i,u_{r-v+i},x_{r-v+i}\}$ for $i=2,\ldots,v$. Note that by \Cref{lem:cwrv}, we have $r-v\geq 1$, and thus, $u_1\neq u_{r-v+1}$. Consequently, we see that $B_i$ induces a $P_3$ in $G$ with $\deg u_i=2$ and $\deg x_i=1$ for each $i=1,\ldots,v$. Therefore, by \Cref{main}, we get $B_i\bigcap B\neq\phi$ for each $i\in [v]$. Since $B_i\bigcap B_j=\phi$ for $i\neq j$, we have 
   $$v\leq\sum_{i=1}^{v}|B\bigcap B_i|\leq |B|=\v(G).$$ 
   Hence, we get $\v(G)=v$.
\end{proof}

\subsection{Future Directions} Let us consider the set $A(n)$ defined in \Cref{thm:RV(n)}. A natural question is how sharp the bound $A(n)$ is for $\RV(n)$. Observe that, in proving $A(n)\subseteq \RV(n)$, we used chordal graphs in both cases of the proof of \Cref{thm:RV(n)}. If $\RV_{\text{chordal}}(n)$ denotes the subset of $\RV(n)$ obtained by restricting to connected chordal graphs, then we have from the proof of \Cref{thm:RV(n)} that
\[
A(n)=\{(r,v)\in \mathbb{N}^{2} \mid 1\leq r<\tfrac{n}{2},\; 1\leq v\leq r-\left\lceil\tfrac{r}{n-2r}\right\rceil+1\}\subseteq \RV_{\text{chordal}}(n).
\]
Motivated by this observation and supported by our experimental evidence, we propose the following conjecture:
\begin{conjecture}
    For the class of connected chordal graphs on $n\geq 3$ vertices, we have
    $$\RV_{chordal}(n)=\left\{(r,v)\in \mathbb{N}^{2} \mid 1\leq r<\frac{n}{2}, 1\leq v\leq r-\left\lceil\frac{r}{n-2r}\right\rceil+1\right\}.$$
\end{conjecture}

As a direction for future work, let $\RV_{\text{bip}}(n)$ denote the collection of all pairs (regularity, $\v$-number) arising from connected bipartite graphs on $n$ vertices. Observe that knowing $\RV(n)$ does not necessarily determine the set $\RV_{\text{bip}}(n)$. Motivated by the work of Erey and Hibi~\cite{eh22}, it would be desirable to determine the set $\RV_{\text{bip}}(n)$ explicitly. Once this case is better understood, one may then investigate the larger set $\RV(n)$, which is quite challenging. Furthermore, it would also be interesting to compute these pairs for some well-known classes of connected graphs.

\subsection*{Acknowledgements} Saha acknowledges support from the Anusandhan National Research Foundation (ANRF), Government of India, under the ARG-MATRICS Grant No. ANRF/ARGM/2025/002203/MTR.

\subsection*{Data availability statement} Data sharing does not apply to this article as no new data were created or
analyzed in this study.

\subsection*{Conflict of interest} The authors declare that they have no known competing financial interests or personal
relationships that could have appeared to influence the work reported in this paper.

\bibliographystyle{abbrv}
	\bibliography{refrv}

\end{document}